\documentclass[12pt,reqno]{amsart}
\usepackage{a4wide}
\usepackage[T1]{fontenc}
\usepackage{amssymb,amsmath,amsthm,latexsym}
\usepackage{mathrsfs}
\usepackage[usenames,dvipsnames]{color}
\usepackage{graphicx}
\usepackage{hyperref}
\usepackage{mdwlist}
\usepackage{enumerate}
\usepackage{enumitem}
\usepackage{mathtools,dsfont,wasysym}
\usepackage{stmaryrd}
\usepackage{hyperref}
\usepackage{multicol}
\hypersetup{colorlinks=true,  linkcolor=blue, citecolor=red, urlcolor=cyan}

\usepackage[mathscr]{eucal}
\usepackage{faktor} 
\usepackage{tikz-cd}

\usepackage{graphicx}
\graphicspath{ {./images/} }

\usepackage[all]{xy}
\usepackage{txfonts}
\usepackage{bigints}
\usepackage{multirow}
\usepackage{mathtools}
\allowdisplaybreaks

\newcommand{\norm}[1]{\left\lVert#1\right\rVert}

\newcommand{\adef}{\begin{defin}}
\newcommand{\zdef}{\end{defin}}
\newtheorem{defin}{Definition}
 
\newtheorem{theorem}{Theorem}[section]
\newtheorem{lemma}[theorem]{Lemma}
\newtheorem{proposition}[theorem]{Proposition}
\newtheorem{corollary}[theorem]{Corollary}

\numberwithin{equation}{section}

\numberwithin{equation}{section}

\newcommand{\vertiii}[1]{{\left\vert\kern-0.25ex\left\vert\kern-0.25ex\left\vert #1
    \right\vert\kern-0.25ex\right\vert\kern-0.25ex\right\vert}}

\def\block(#1,#2)#3{\multicolumn{#2}{c}{\multirow{#1}{*}{$ #3 $}}}

\DeclarePairedDelimiterX{\inp}[2]{\langle}{\rangle}{#1, #2}

\newcommand{\abs}[1]{\lvert#1\rvert}

\theoremstyle{definition}

\newtheorem{example}[theorem]{Example}

\newcommand{\N}{\mathbb{N}}

\theoremstyle{remark}

\newcommand{\Q}{\mathbb{Q}}
\newcommand{\R}{\mathbb{R}}

\newcommand{\C}{\mathbb{C}}

\newcommand{\K}{\mathbb{K}}
\newcommand{\Ll}{\mathcal{L}}

\numberwithin{equation}{section}

\makeatletter
\renewcommand{\tocsection}[3]{%
	\indentlabel{\@ifnotempty{#2}{\bfseries\ignorespaces#1 #2\quad}}\bfseries#3}
\renewcommand{\tocsubsection}[3]{%
	\indentlabel{\@ifnotempty{#2}{\ignorespaces#1 #2\quad}}#3}

\newcommand\@dotsep{4.5}
\def\@tocline#1#2#3#4#5#6#7{\relax
	\ifnum #1>\c@tocdepth 
	\else
	\par \addpenalty\@secpenalty\addvspace{#2}%
	\begingroup \hyphenpenalty\@M
	\@ifempty{#4}{%
		\@tempdima\csname r@tocindent\number#1\endcsname\relax
	}{%
		\@tempdima#4\relax
	}%
	\parindent\z@ \leftskip#3\relax \advance\leftskip\@tempdima\relax
	\rightskip\@pnumwidth plus1em \parfillskip-\@pnumwidth
	#5\leavevmode\hskip-\@tempdima{#6}\nobreak
	\leaders\hbox{$\m@th\mkern \@dotsep mu\hbox{.}\mkern \@dotsep mu$}\hfill
	\nobreak
	\hbox to\@pnumwidth{\@tocpagenum{\ifnum#1=1\bfseries\fi#7}}\par
	\nobreak
	\endgroup
	\fi}
\AtBeginDocument{%
	\expandafter\renewcommand\csname r@tocindent0\endcsname{0pt}
}
\def\l@subsection{\@tocline{2}{0pt}{2.5pc}{5pc}{}}
\makeatother

\begin{document}

	\title[Twisted Hilbert spaces defined by bi-Lipschitz maps]{Twisted Hilbert spaces defined by bi-Lipschitz maps}

		\author[Corrêa]{Willian Corrêa }
	\address[Corrêa]{Departamento de Matemática, Instituto de Ciências Matemáticas e de Computação, Universidade de São Paulo, Avenida Trabalhador São-carlense, 400 -
Centro, CEP: 13566-590, São Carlos, SP. \newline
		\href{https://orcid.org/0000-0003-2172-4019}{ORCID: \texttt{0000-0003-2172-4019} }}
	\email{\texttt{willhans@icmc.usp.br}}

		\author[Dantas]{Sheldon Dantas}
	\address[Dantas]{Departamento de Análisis Matemático, Facultad de Ciencias Matemáticas, Universidad de Valencia, Doctor Moliner 50, 46100 Burjasot (Valencia), Spain. \newline
		\href{http://orcid.org/0000-0001-8117-3760}{ORCID: \texttt{0000-0001-8117-3760} } }
	\email{\texttt{sheldon.dantas@uv.es}}

\author[Rodríguez-Vidanes]{Daniel L. Rodríguez-Vidanes}
	\address[Rodríguez-Vidanes]{Grupo de Investigación: Análisis Matemático y Matemática Aplicada, Departamento de Matemática Aplicada a la Ingeniería Industrial, Escuela Técnica Superior de Ingeniería y Diseño Industrial, Ronda de Valencia 3, Universidad Politécnica de Madrid,
		Madrid, 28012, Spain \newline
		\href{http://orcid.org/0000-0000-0000-0000}{ORCID: \texttt{0000-0003-2240-2855} }}
	\email{\texttt{dl.rodriguez.vidanes@upm.es}}

	\begin{abstract} We obtain an infinite-dimensional cone of singular twisted Hilbert spaces $Z(\varphi)$ which are isomorphic to their duals but not to their conjugate duals. We do that by showing that the subset of all bi-Lipschitz maps from $[0, \infty)$ to $\R$ is coneable. We also provide a characterization of the Kalton-Peck space among all twisted Hilbert spaces of the form $Z(\varphi)$, which gives a partial answer to a conjecture of F. Cabello S\'anchez and J. Castillo.
	\end{abstract}

	\thanks{ }
	
	
	\maketitle

	\thispagestyle{plain}


\section{Introduction}

The Kalton-Peck space is a remarkable example of a Banach space with exotic properties. It was the second example of a nontrivial (separable) twisted Hilbert space, i.e., a Banach space $X$ containing an isomorphic copy of $\ell_2$ such that $X/\ell_2$ is also isomorphic to $\ell_2$ and $X$ non-Hilbertian. The first example was given by P. Enflo, J. Lindenstrauss and G. Pisier in \cite{ELP} and we will denote it by $ELP$, for short. The construction of $ELP$ is local in nature, in the sense that it is of the form $\ell_2(F_n)$, where each $F_n$ is a finite dimensional Banach space that are less and less Hilbertian. One of the consequences of $ELP$'s local nature is that, even though it is a nontrivial twisted Hilbert space, $ELP$ is full of complemented copies of $\ell_2$ (\cite[Section 4]{CMS}, or \cite[Section 10.9]{Homological}; see also \cite{Castillo} for a relevant reference to the general theory and open problems). On the other hand, the quotient map $q : Z_2 \rightarrow \ell_2$ is strictly singular and $Z_2^* \cong Z_2$, and therefore, the inclusion $i : \ell_2 \rightarrow Z_2$ is strictly cosingular. The space $Z_2$ is then an extreme example of a twisted Hilbert space.

More generally, a \emph{twisted sum} of Banach spaces $Y$ and $Z$ is a Banach space $X$ containing an isomorphic copy of $Y$ such that the corresponding quotient is isomorphic to $Z$. We represent this by using a short exact sequence
\[
\xymatrix{0 \ar[r] & Y \ar[r]^i & X \ar[r]^q & Z \ar[r] & 0.}
\]
The twisted sum $X$ is said to be \emph{nontrivial} if the copy of $Y$ is not complemented in $X$. If $q$ is strictly singular, then $X$ is said to be a \emph{singular twisted sum}.

A great deal of work has been done in trying to determine the singularity of twisted sums (see, for instance, \cite{Cabello_Factorization, Singular, CMS, TypeCotype, JesusKalton-Peck, Cabello_no_singular, Cabello_sso}). The present paper is a contribution to the family of examples of singular twisted Hilbert spaces. Actually, the spaces we present will have the following stronger property: let $\ell_{\Phi}$ be the Orlicz sequence space corresponding to a convex function which is equivalent to $t^2 \log^2 t$ near $0$. Then any normalized basic sequence in $Z_2$ has a subsequence equivalent to the canonical basis of either $\ell_2$ or $\ell_{\Phi}$. We call that property \emph{$\Phi$-stability}




In \cite{KP}, Kalton and Peck show that, given an unbounded Lipschitz map $\varphi : [0, \infty) \rightarrow \R$, one can obtain a twisted Hilbert space $Z(\varphi)$ and the Kalton-Peck space $Z_2$ corresponds to the choice $\varphi(t) = ct$ for any $c \neq 0$. Showing that there is an infinite-dimensional cone of bi-Lipschitz functions from $[0, \infty)$ to $\R$, we obtain an ``infinite-dimensional cone'' of singular twisted Hilbert spaces all of which are $\Phi$-stable. We conclude the paper by providing a characterization of the Kalton-Peck space among all spaces $Z(\varphi)$ similar to the the characterization of $c_0$ and $\ell_p$ ($1 \leq p < \infty$) as the only Banach spaces with a basis which is equivalent to all of its normalized block bases. This partially answers a conjecture by Cabello Sánchez and Castillo in \cite[page 517]{Homological}.

\section{Twisted sums induced by Lipschitz maps}\label{Sec:Background}

Let $\mathbb{K}$ be either the field of real or complex numbers. Recall that a \emph{quasi-norm} on a $\mathbb{K}$-vector space $X$ is a function $\|\cdot\| : X \rightarrow [0, \infty)$ satisfying the following properties:
\begin{enumerate}
    \item $\|x\| = 0 \iff x = 0$,
    \item $\|\lambda x\| = \abs{\lambda} \|x\|$ for every $x \in X$, $\lambda \in \K$ and
    \item there is a constant $C \geq 1$ such that for every $x, y \in X$, we have $\|x + y\| \leq C (\|x\| + \|y\|)$.
\end{enumerate}
The space $X$ endowed with a quasi-norm is called a \emph{quasi-normed space}. The Aoki-Rolewicz theorem says that there are $p \in (0, 1]$ and an equivalent quasi-norm $|||\cdot|||$ on $X$ such that $d(x, y) = |||x - y|||^p$ defines a distance on $X$. If that distance is complete, we say that $X$ is a \emph{quasi-Banach space} and that the original quasi-norm is complete. A twisted sum of $Y$ and $Z$ is a quasi-normed space $X$ such that it contains an isomorphic copy of $Y$ and the respective quotient is isomorphic to $Z$. It is possible to prove that $X$ is a Banach space. If $i(Y) \subset X$ is not complemented, we say that $X$ is a \emph{nontrivial twisted sum}. If $Y = Z = \ell_2$, then we say that $X$ is a (separable) \emph{twisted Hilbert space}. If $i$ is strictly cosingular, we say that $X$ is a \emph{cosingular twisted sum}. If $q$ is strictly singular, then $X$ is said to be a \emph{singular twisted sum}.

In \cite{KP}, Kalton and Peck provide a whole factory of twisted Hilbert spaces by making use of Lipschitz maps. If $\varphi : [0, \infty) \rightarrow \K$ is a Lipschitz map with $\varphi(0) = 0$, we define $\Omega_{\varphi} : c_{00} \rightarrow \ell_2$ by
\[
\Omega_{\varphi}(x) = \sum_{n=1}^\infty x_n \varphi\left(-\log \frac{\abs{x_n}}{\|x\|_2}\right) e_n 
\]
where $(e_n)$ is the canonical basis of $\ell_2$, with the understanding that $\Omega_{\varphi}(0) = 0$. In this case, we have that $\Omega_{\varphi}$ is a \emph{quasi-linear map}. This means that it is homogeneous and there is $C > 0$ such that
\[
\|\Omega_{\varphi}(x + y) - \Omega_{\varphi}(x) - \Omega_{\varphi}(y)\|_2 \leq C(\|x\|_2 + \|y\|_2).
\]
This implies that if we define $\|\cdot\|_{\varphi} : \ell_2 \times c_{00} \rightarrow [0, \infty)$ by
\[
\|(y, x)\|_{\varphi} := \|y - \Omega_{\varphi}(x)\|_2 + \|x\|_2
\]
then we have that $\|\cdot\|_{\varphi}$ is a quasi-norm. If we denote by $Z(\varphi)$ the completion of $(\ell_2 \times c_{00}, \|\cdot\|_{\varphi})$, then we obtain a short exact sequence
\[
\xymatrix{0 \ar[r] & \ell_2 \ar[r]^i & Z(\varphi) \ar[r]^q & \ell_2 \ar[r] & 0}
\]
where $i(y) = (y, 0)$ and $q(y, x) = x$. If $\varphi$ is unbounded, then $Z(\varphi)$ is a nontrivial twisted Hilbert space, i.e., $i(\ell_2)$ is not complemented in $Z(\varphi)$. 


Let $X_1$ and $X_2$ be two twisted sums of $Y$ and $Z$. We say that $X_1$ and $X_2$ are \emph{equivalent} if there is a commutative diagram
\[
\xymatrix{0 \ar[r] & Y \ar[r]\ar@{=}[d] & X_1 \ar[r]\ar[d]^T & Z \ar[r]\ar@{=}[d] & 0 \\
0 \ar[r] & Y \ar[r] & X_2 \ar[r] & Z \ar[r] & 0}
\]
where $T$ is an isomorphism. The exact sequences at $X_1$ and $X_2$ are said to be \emph{projectively equivalent} if there is a commutative diagram
\[
\xymatrix{0 \ar[r] & Y \ar[r]\ar[d]^{\alpha Id_Y} & X_1 \ar[r]\ar[d]^T & Z \ar[r]\ar[d]^{\beta Id_Z} & 0 \\
0 \ar[r] & Y \ar[r] & X_2 \ar[r] & Z \ar[r] & 0}
\]
where $T$ is an isomorphism and $\alpha, \beta \neq 0$. The Lipschitz maps $\varphi, \psi : [0, \infty) \rightarrow \K$ with $\varphi(0) = \psi(0) = 0$ are said to be \emph{equivalent} whenever
\[
\sup_{t > 0} \left|\varphi(t) - \psi(t)\right| < \infty.
\]
The maps $\varphi$ and $\psi$ are \emph{projectively equivalent} (or, which is the same, that $\Omega_\varphi$ and $\Omega_\psi$ are projectively equivalent) if $\varphi$ is equivalent to $c\psi$, for some $c \neq 0$. The twisted sums $Z(\varphi)$ and $Z(\psi)$ are (projectively) equivalent if and only if $\varphi$ and $\psi$ are (projectively) equivalent. If we have a twisted Hilbert space  
\[
\xymatrix{0 \ar[r] & \ell_2 \ar[r]^i & X \ar[r]^q & \ell_2 \ar[r] & 0}
\]
then its dual is also a twisted Hilbert space with
\[
\xymatrix{0 \ar[r] & \ell_2 \ar[r]^{q^*} & X^* \ar[r]^{i^*} & \ell_2 \ar[r] & 0.}
\]
In the case of twisted Hilbert spaces induced by Lipschitz maps, the dual of $Z(\varphi)$ is equivalent to $Z(-\varphi)$ and the duality action is given by $\langle (x, y) , (w, z) \rangle = \langle y, w \rangle + \langle x, z \rangle$ (for a more general result regarding twisted sums induced by interpolation, see \cite[Section 5]{HigherOrderUs}). In particular, $Z(\varphi)$ is isomorphic to $Z(\varphi)^*$. The sequence
\[
\mathcal{E} := \{(e_1, 0), (0, e_1), (e_2, 0), (0, e_2), \ldots\}
\]
is a normalized basis for $Z(\varphi)$, which is unconditional if and only if $Z(\varphi)$ is trivial (that is, Hilbertian) \cite{Kalton_Unconditional}. On the other hand, if we let $E_n = [(e_n, 0), (0, e_n)]$, then $(E_n)_n$ defines a $2$-dimensional UFDD of $Z(\varphi)$ (as in \cite[page 30]{KP}). Another fact (that follows immediately from the boundedness of the function $t \mapsto t \log t$ on $(0, 1)$) that we will use in the paper is that $\Omega_{\varphi}: \ell_2 \rightarrow c_{00}$ is bounded.

\section{Twisted sums induced by bi-Lipschitz maps}\label{sec:bi_lip}

In this section, we show that whenever $\varphi$ is a bi-Lipschitz map, then $Z(\varphi)$ is very similar to the Kalton-Peck space $Z_2$. In order to prove the results in this section, we use similar arguments to the ones from \cite{KP}. Nevertheless, we present all the details for the sake of completeness.

Let us recall that, given two metric spaces $X$ and $Y$, a map $f: X \rightarrow Y$ is said to be {\it Lipschitz} if there exists a constant $a \geq 0$ such that $d(f(x),f(y)) \leq a \cdot d(x,y)$ for every $x,y \in X$. If there is also $b > 0$ such that $d(f(x), f(y)) \geq b \cdot d(x,y)$ for all $x,y \in X$, then we will call $f$ a {\it bi-Lipschitz} map. We denote by $\Ll$ the class of all Lipschitz maps $\varphi : [0, \infty) \rightarrow \K$ with $\varphi(0) = 0$ such that $\varphi(t) \rightarrow \infty$ as $t \rightarrow \infty$ and by $\Ll_{bi}$ the subclass of $\Ll$ consisting of all bi-Lipschitz maps.


\begin{lemma}\label{lem:Orlicz_subspace}
    Let $\varphi \in \Ll_{bi}$. Consider $(v_n)$ to be a normalized block sequence in $\ell_2$ and let $w_n = (\Omega_{\varphi}(v_n), v_n) \in Z(\varphi)$. Then, we have that $(w_n)$ is a basic sequence equivalent to the basis of the Orlicz sequence space $\ell_{\Phi}$.
\end{lemma}
\begin{proof}
Let $(w_n)_{n \geq 1}$ be a block sequence of $\mathcal{E}$. Write $v_n = \sum\limits_{k=l_{n-1}+1}^{l_n} v_n(k) e_k$ and let $t \in c_{00}$. We have that 
\begin{eqnarray*}
\norm{\sum_{n=1}^\infty t_n w_n}_{\varphi} 
& = & \norm{\sum_{n=1}^\infty \sum\limits_{k=l_{n-1}+1}^{l_n} t_n v_n(k) \left(\varphi\left(\log \frac{1}{\abs{v_n(k)}}\right)  - \varphi\left(\log \frac{\|t\|_2}{\abs{t_n v_n(k)}}\right)\right)}_2 + \norm{t}_2.
\end{eqnarray*}
Being a twisted sum of $\ell_2$, we have that $Z(\varphi)$ is also reflexive (since being reflexive is a three space property \cite{3SP}) and then $(w_n)$ is boundedly complete. Therefore, given a sequence of scalars $(t_n)$, we have that $\sum t_n w_n$ converges in $Z(\varphi)$ if and only if
\[
\sup_N \norm{\sum_{n=1}^N \sum\limits_{k=l_{n-1}+1}^{l_n} t_n v_n(k) \left(\varphi\left(\log \frac{1}{\abs{v_n(k)}}\right)  - \varphi\left(\log \frac{\|t\|_2}{\abs{t_n v_n(k)}}\right)\right)}_2 + \|t\|_2 < \infty.
\]
Since $\varphi$ is bi-Lipschitz, we have that 
\[
\left(\sum_{n=1}^\infty \abs{t_n}^2 \left|\log \frac{\|t\|_2}{\abs{t_n}} \right|^2 \right)^{\frac{1}{2}} + \|t\|_2 < \infty,
\]
which means in turn that $(w_n)$ is equivalent to the canonical basis of $\ell_{\Phi}$.
\end{proof}

Recall from the Introduction that the Kalton-Peck space is $\Phi$-stable.

\begin{theorem}\label{thm:good_blocks}
Let $\varphi \in \mathcal{L}_{bi}$. Then $\ell_{\varphi}$ is $\Phi$-stable. More precisely, every normalized basic sequence $(w_n)$ in $Z(\varphi)$ has a subsequence equivalent to the canonical basis of either $\ell_2$ or $\ell_{\Phi}$. In particular, if $w_n = (u_n, v_n)$ and $\inf \|v_n\| > 0$, then $(w_n)$ has a subsequence equivalent to the canonical basis of $\ell_{\Phi}$.
\end{theorem}

By using \cite[Proposition 6.9]{Singular}, we have that whenever $\varphi$ is expansive, $Z(\varphi)$ is singular. Theorem \ref{thm:good_blocks} gives us an alternative proof.


\begin{theorem}\label{thm:singularity}
Let $\varphi \in \mathcal{L}_{bi}$. Then, $Z(\varphi)$ is a singular and cosingular twisted Hilbert space.
\end{theorem}
\begin{proof}
Let $q : Z(\varphi) \rightarrow \ell_2$ be the quotient map and assume that it is not strictly singular. So there is a subspace $W \subset Z(\varphi)$ such that $q|_W$ is an isomorphism onto its image $q(W)$. Since $q(W) \subset \ell_2$, it has an orthonormal basis $f_n$. It follows $w_n = (q|_W)^{-1}(f_n)$ is a basic sequence in $Z(\varphi)$ equivalent to the canonical basis of $\ell_2$ but $\inf \|q(w_n)\| > 0$, which contradicts Theorem \ref{thm:good_blocks}. The cosingularity now follows from the duality.
\end{proof}

\vspace{0.2cm}
To finish this section, it is worth recalling the alternative representation of $Z_2$ as a twisted sum
\[
\xymatrix{0 \ar[r] & \ell_{\Phi} \ar[r] & Z_2 \ar[r] & \ell_{\Phi}^* \ar[r] & 0.}
\]
For this, we refer to \cite[Section 5.3]{CCFG}, \cite[Proposition 3.6]{Symmetries} and \cite[Facts 1.1]{OperatorsZ2}. See also \cite{Quasilinear_Inversion}. Similarly, we have:
\begin{proposition}
Let $\varphi \in \mathcal{L}_{bi}$. Then, we have the short exact sequence
\[
\xymatrix{0 \ar[r] & \ell_{\Phi} \ar[r]^j & Z(\varphi) \ar[r]^p & \ell_{\Phi}^* \ar[r] & 0}
\]
where $j(x) = (0, x)$ and $p(y, x) = y$.
\end{proposition}

\section{Distinguishing twisted Hilbert spaces}\label{sec:distinguish}

The motivation for this section is to establish a simple criterion to guarantee that $Z(\varphi)$ is non-isomorphic to $Z(\psi)$. We will follow the proof that $Z(\varphi_{\alpha})$ is not isomorphic to a subspace of $Z(\varphi_{\beta})$ presented in \cite[Chapter~16]{Benyamini1998geometric}, which is a simplification of the proof of \cite{Kalton_Elementary}. We present the proof in full detail for the sake of completeness.

Let $\varphi \in \Ll_{bi}$. Let us consider the following classes of functions. We write
\begin{enumerate}
    \item $\varphi \in \Ll_{bis}$ if the second derivative of $\varphi$ exists and goes to $0$ as $t$ goes to $\infty$.

    \item $\varphi \in \Ll_{bid}$ if
        \[
            \limsup_{n, m \rightarrow \infty}
            \left|\frac{\varphi(\log n) - \varphi(\log n^{1/2})}{\log n^{1/2}} - \frac{\varphi(\log m) - \varphi(\log m^{1/2})}{\log m^{1/2}}\right| > 0.
        \]
\end{enumerate}
We then set $\Ll_{bids} := \Ll_{bis} \cap \Ll_{bid}$.
\vspace{0.2cm} 

Our main goal here is to prove the following result. Recall that two Banach spaces $X$ and $Y$ are said to {\it incomparable} if no infinite dimensional subspace of $X$ is isomorphic to a subspace of $Y$.

\begin{theorem}\label{thm:distinguishing}
Let $\varphi \in \Ll_{bid}$ and $\psi \in \Ll_{bis}$. If $\Omega_\varphi$ and $\Omega_\psi$ are not projectively equivalent, then we have that $Z(\varphi)$ is incomparable to $Z(\psi)$.
\end{theorem}

Recall that $\Omega_{\varphi}$ and $\Omega_{\psi}$ are projectively equivalent if and only if
\[
\sup_{t > 0} \; \abs{\varphi(t) - a \psi(t)} < \infty
\]
for some $a \neq 0$ \cite{KP}. 
Therefore, we have a simple test for projective equivalence.
Each element of $Z(\varphi)$ is given by a pair $(x, y)$ of sequences. Therefore, it is also natural to consider $2\times 2$ matrices as acting on $Z(\varphi)$ (see also \cite{OperatorsZ2}). More precisely, we write 
\[
\begin{pmatrix}
\lambda & \mu \\
\eta & \sigma
\end{pmatrix} (x, y) = (\lambda x + \mu y, \eta x + \sigma y).
\]

For every $n \geq 1$, we set $f_n := \sum\limits_{j=1}^n e_j \in \ell_2$.

\begin{lemma}
Let $T = \begin{pmatrix}
\lambda & \mu \\
\eta & \sigma
\end{pmatrix} : Z(\varphi) \rightarrow Z(\psi)$ be bounded. If $\varphi, \psi \in \mathcal{L}_{bi}$ are not projectively equivalent, then $\lambda = \eta = \sigma = 0$.
\end{lemma}
\begin{proof}
Let $u_n = n^{-\frac{1}{2}} (f_n, 0)$. Then, $Tu_n = n^{-\frac{1}{2}} (\lambda f_n, \eta f_n)$, and therefore
\begin{eqnarray*}
\|T\| & \geq & \|T u_n\|_{\psi} \\
& = & n^{-\frac{1}{2}} (\|\lambda f_n - \eta \Omega_{\psi}f_n\|_2 + \|\eta f_n\|_2) \\
& \geq & n^{-\frac{1}{2}} (\abs{\eta} \|\Omega_{\psi} f_n\|_2 + (\abs{\eta} - \abs{\lambda}) \|f_n\|_2) \rightarrow \infty
\end{eqnarray*}
if $\eta \neq 0$, since $Z(\psi)$ is nontrivial.

Now take $u_n = n^{-\frac{1}{2}} (\Omega_{\varphi}f_n, f_n)$. Then, $Tu_n = n^{-\frac{1}{2}}(\lambda \Omega_{\varphi}e_n + \mu f_n, \sigma f_n)$ (since $\eta = 0$) and
\begin{eqnarray*}
\|T\| & \geq & \|Tu_n\|_{\psi} \\
      & \geq & n^{-\frac{1}{2}} \|\lambda \Omega_{\varphi} f_n + \mu f_n - \sigma \Omega_{\psi}f_n\|_2 \\
      & \geq & n^{-\frac{1}{2}}(\|\lambda \Omega_{\varphi} f_n - \sigma \Omega_{\psi}f_n\|_2 - \|\mu f_n\|_2) \rightarrow \infty
\end{eqnarray*}
if $\Omega_{\varphi}$ and $\Omega_{\psi}$ are not projectively equivalent and $\lambda \neq 0$ or $\sigma \neq 0$. Indeed, we have that
\begin{eqnarray*}
\|\lambda \Omega_{\varphi}(f_n) - \sigma \Omega_{\psi}(f_n)\| & = & \sqrt{n} \left|\lambda \varphi(\log \sqrt{n}) - \mu \psi (\log \sqrt{n})\right|.
\end{eqnarray*}
Let $\varepsilon>0$ be given. If we go far enough down in the real line, then we may approximate any real number by one of the form $\log \sqrt{n}$ up to $\varepsilon$,
Since $\varphi$ and $\psi$ are bi-Lispchitz, we obtain
\begin{equation*} 
\sup_{x > 0} \left|\lambda \varphi(\log x) - \mu \psi (\log x)\right| < \infty \iff \sup_n \left|\lambda \varphi(\log \sqrt{n}) - \mu \psi (\log \sqrt{n})\right| < \infty
\end{equation*}
as we wanted to prove.
\end{proof}


\begin{proposition}Let $\varphi \in \mathcal{L}_{bid}$ and $\psi \in \mathcal{L}_{bis}$. If $Z(\varphi)$ is isomorphic to a subspace of $Z(\psi)$, then there is a bounded $T = \begin{pmatrix}
\lambda & \mu \\
\eta & \sigma
\end{pmatrix} : Z(\varphi) \rightarrow Z(\psi)$ with $\sigma \neq 0$.
\end{proposition}

\begin{proof}
Let $S : Z(\varphi) \rightarrow Z(\psi)$ with $\|S\| \leq 1$ and $\|S(e_n, 0)\|_{\psi} \geq c>0$ for every $n$. By taking a subsequence and perturbing $S$, we may suppose that the images of $(e_n, 0)$ and $(0, e_n)$ are blocks of $\mathcal{E}$, that is, that there are indexes $p_1 < p_2 < p_3 < \cdots$ such that if $I_n = \{k : p_n \leq k < p_{n+1}\}$, then $S(e_n, 0) = (u_n, y_n)$ and $S(0, e_n) = (v_n, w_n)$, with $u_n, y_n, v_n, w_n \in [e_k]_{k \in I_n}$.

\vspace{0.2cm}
\noindent
{\bf Claim 1}:
$\|y_n\|_2 \rightarrow 0$.
\vspace{0.2cm}

Intuitively, this means that the image by $S$ of the canonical copy of $\ell_2$ in $Z(\varphi)$ is essentially contained in the canonical copy of $\ell_2$ in $Z(\psi)$.

\begin{proof}[Proof of Claim 1:]
Assume, by means of contradiction, that $\|y_n\|_2 \not\rightarrow 0$. Furthermore, assume without loss of generality that there is $\delta>0$ such that $\|y_n\|_2 > \delta$ for every $n \in \N$. Since $\|(u_n, y_n)\|_{\psi} = \|S(e_n, 0)\|_{\psi} \leq 1$, it yields
\begin{equation}\label{eq:1}
\sum\limits_{k \in I_n} \left|u_n(k) - y_n(k)\psi\left(\log \frac{\|y_n\|_2}{\abs{y_n(k)}}\right)\right|^2 \leq 1.
\end{equation}
Also, since $\left\|\left(\sum\limits_{n=1}^N u_n, \sum \limits_{n=1}^N y_n\right)\right\|_{\psi} = \|S(f_N, 0)\|_{\psi} \leq N^{-\frac{1}{2}}$ for each $N$, we have that
\begin{equation}\label{eq:2}
\sum\limits_{n=1}^N \sum\limits_{k \in I_n} \left|u_n(k) - y_n(k)\psi\left(\log \frac{\left\|\sum \limits_{n=1}^N y_n\right\|_2}{\abs{y_n(k)}}\right)\right|^2 \leq N.
\end{equation}
Using \eqref{eq:1}, \eqref{eq:2} and the orthogonality of the vectors under consideration, we obtain
\[
\sum\limits_{n=1}^N \sum_{k \in I_n} \abs{y_n(k)}^2 \left|\psi\left(-\log \frac{\|y_n\|_2}{\left\|\sum\limits_{j=1}^N y_n\right\|_2}\right)\right|^2 \leq C_1 N.
\]
Now since $\psi$ is bi-lispchitz and $\delta<\|y_n\|_2 \leq 1$, we have that
\[
C_2 N \delta^2 \log^2 (\delta N^{\frac{1}{2}}) \leq C_1 N
\]
which yields a contradiction.
\end{proof}

Once again, by passing to a subsequence and perturbing $S$ we may suppose that $y_n = 0$ for every $n$.

\vspace{0.2cm}
\noindent
{\bf Claim 2}:
$\|w_n\|_{\infty} \not\rightarrow 0$. 
\vspace{0.2cm}

Intuitively, this shows that the second coordinate is really used by $S$ when applied to the vectors $(0, e_n)$.

\begin{proof}[Proof of Claim 2:]
Otherwise, let $K > 0$ and $n_0\in \mathbb N$ be such that $\|w_n\|_{\infty} \leq e^{-K}$ for each $n > n_0$. Let $m_1, m_2 \in \N$, $M = 4 \max\{m_1^2, m_2^2\}$, $\sigma_1 = \frac{1}{2} \log m_1$ and $\sigma_2 = \frac{1}{2} \log m_2$. Also let $J_{M} = \{k : n_0 + 1 \leq k \leq n_0 + M\}$ and take $A \subset J_{M}$ of cardinality $m_1$. 
Note that we have $\left\|\left(\sum\limits_{n \in A} \varphi(\sigma_1) e_n, \sum\limits_{n \in A} e_n\right)\right\|_{\varphi} = m_1^{1/2}$, so applying $S$ we obtain
\begin{equation}\label{eq:3}
\sum\limits_{n \in A} \sum\limits_{k \in I_n} \left|\varphi(\sigma_1) u_n(k) + v_n(k) - w_n(k) \psi\left(\log \frac{\|\sum_{n \in A} w_n\|_2}{\abs{w_n(k)}}\right)\right|^2 \leq m_1.
\end{equation}
Since $\psi$ is Lipschitz, we also have that
\begin{eqnarray*}
\sum\limits_{n \in A} \sum\limits_{k \in I_n} \left|w_n(k) \psi\left(\log \frac{\left\|\sum_{n \in A} w_n\right\|_2}{\abs{w_n(k)}}\right) - w_n(k) \psi\left(\log \frac{m_1^{1/2}}{\abs{w_n(k)}}\right)\right|^2
\end{eqnarray*}
is bounded above by
\begin{eqnarray*}
L(\psi)^2 \left\|\sum_{n \in A} w_n\right\|_2^2 \log^2 \frac{\left\|\sum_{n \in A} w_n\right\|_2}{m_1^{1/2}} = L(\psi)^2 \frac{m_1}{m_1} \left\|\sum_{n \in A} w_n\right\|_2^2 \log^2 \frac{\left\|\sum_{n \in A} w_n\right\|_2}{m_1^{1/2}}
\end{eqnarray*}
so that
\begin{equation}\label{eq:4}
\sum\limits_{n \in A} \sum\limits_{k \in I_n} \left|w_n(k) \psi\left(\log \frac{\|\sum_{n \in A} w_n\|}{\abs{w_n(k)}}\right) - w_n(k) \psi\left(\log \frac{m_1^{1/2}}{\abs{w_n(k)}}\right)\right|^2 \leq L(\psi)^2 \frac{m_1}{e^2}.
\end{equation}
The last inequality comes from the fact that the function $t \mapsto t \log t$ is bounded by $\frac{1}{e}$ on $(0, 1]$, using that $\|w_n\|_2 \leq \|S(0, e_n)\|_{\psi} \leq 1$ and that the vectors $(w_n)$ are orthogonal.

Now, by \eqref{eq:3} and \eqref{eq:4} (looking at the corresponding vectors in $\ell_2$ and using the triangle inequality), we have
\[
\sum\limits_{n \in A} \sum\limits_{k \in I_n} \abs{\varphi(\sigma_1) u_n(k) + v_n(k) - w_n(k) \psi(\sigma_1 - \log \abs{w_n(k)})}^2 \leq C m_1
\]
for some $C>0$.
Taking the average over all possible $A$ we obtain
\begin{equation}\label{eq:5}
\sum\limits_{n = n_0+1}^{n_0+M} \sum\limits_{k \in I_n} \abs{\varphi(\sigma_1) u_n(k) + v_n(k) - w_n(k) \psi(\sigma_1 - \log \abs{w_n(k)})}^2 \leq C M.
\end{equation}
Now we repeat the calculations with $m_1^2$ instead of $m_1$ to arrive at
\begin{equation}\label{eq:6}
\sum\limits_{n = n_0+1}^{n_0+M} \sum\limits_{k \in I_n} \abs{\varphi(2\sigma_1) u_n(k) + v_n(k) - w_n(k) \psi(2\sigma_1 - \log \abs{w_n(k)})}^2 \leq C M.
\end{equation}

Combining \eqref{eq:5} and \eqref{eq:6}, and using the triangle inequality in $\ell_2$, we have
\begin{equation}\label{eq:7}
\sum\limits_{n = n_0+1}^{n_0+M} \sum\limits_{k \in I_n} \abs{(\varphi(2\sigma_1) - \varphi(\sigma_1)) u_n(k) - w_n(k) (\psi(2\sigma_1 - \log \abs{w_n(k)}) - \psi(\sigma_1 - \log \abs{w_n(k)}))}^2 \leq 2C M.
\end{equation}
Now consider the Taylor estimate
\[
\abs{\psi(t) - \psi(s) - \psi'(s)(t-s)} = \frac{1}{2}\abs{\psi''(c)} (t-s)^2
\]
for some $c$ between $t$ and $s$. Apply it to
\[
\psi(2\sigma_1 - \log \abs{w_n(k)}) - \psi(-\log \abs{w_n(k)}) - \psi'(-\log \abs{w_n(k)}) 2\sigma_1
\]
and to
\[
\psi(\sigma_1 - \log \abs{w_n(k)}) - \psi(-\log \abs{w_n(k)}) - \psi'(-\log \abs{w_n(k)}) \sigma_1
\]
and use that $\psi \in \Ll_{bis}$ to obtain that there is $h(K)$ that goes to $0$ as $K \rightarrow \infty$ such that
\begin{equation}\label{eq:8}
\abs{\psi(2\sigma_1 - \log \abs{w_n(k)}) - \psi(\sigma_1 - \log \abs{w_n(k)}) - \psi'(-\log \abs{w_n(k)} )\sigma_1} \leq h(K) \sigma_1^2.
\end{equation}
Now, since $\|w_n\|_2 \leq 1$ for every $n$, we also obtain that
\begin{equation}\label{eq:9}
\sum\limits_{n=n_0+1}^{n_0+M} \sum\limits_{k \in I_n} \abs{w_n(k)}^2 \abs{\psi(2\sigma_1 - \log \abs{w_n(k)}) - \psi(\sigma_1 - \log \abs{w_n(k)}) - \psi'(-\log \abs{w_n(k)} )\sigma_1}^2 \leq M h(K)^2 \sigma_1^4
\end{equation}
Combining \eqref{eq:7}, \eqref{eq:8} and \eqref{eq:9}, and using the triangle inequality in $\ell_2$, we have
\begin{equation}\label{eq:10}
\sum\limits_{n = n_0+1}^{n_0+M} \sum\limits_{k \in I_n} \left|\frac{\varphi(2\sigma_1) - \varphi(\sigma_1)}{\sigma_1} u_n(k) - w_n(k) \psi'(-\log \abs{w_n(k)})\right|^2 \leq M\left( \frac{2C}{\sigma_1^2} + h(K)^2 \sigma_1^2\right).
\end{equation}
Now we repeat the calculations with $m_2$ to obtain
\begin{equation}\label{eq:11}
\sum\limits_{n = n_0+1}^{n_0+M} \sum\limits_{k \in I_n} \left|\frac{\varphi(2\sigma_2) - \varphi(\sigma_2)}{\sigma_2} u_n(k) - w_n(k) \psi'(-\log \abs{w_n(k)})\right|^2 \leq M\left( \frac{2C}{\sigma_2^2} + h(K)^2 \sigma_2^2\right).
\end{equation}
From \eqref{eq:10} and \eqref{eq:11}, we obtain
\[
\sum\limits_{n = n_0+1}^{n_0+M} \sum\limits_{k \in I_n} \abs{u_n(k)}^2 \left|\frac{\varphi(2\sigma_1) - \varphi(\sigma_1)}{\sigma_1} - \frac{\varphi(2\sigma_2) - \varphi(\sigma_2)}{\sigma_2}\right|^2 \leq M\left(\frac{2C}{\sigma_1^2} + \frac{2C}{\sigma_2^2} + h(K)^2(\sigma_1^2 + \sigma_2^2)\right).
\]
Since 
$\|u_n\|_2 \geq c$ for every $n$, we obtain
\[
c^2 M \left|\frac{\varphi(2\sigma_1) - \varphi(\sigma_1)}{\sigma_1} - \frac{\varphi(2\sigma_2) - \varphi(\sigma_2)}{\sigma_2}\right|^2 \leq M\left(\frac{2C}{\sigma_1^2} + \frac{2C}{\sigma_2^2} + h(K)^2(\sigma_1^2 + \sigma_2^2)\right).
\]
So, as $K$ is arbitrary, we arrive at
\[
c^2 \left|\frac{\varphi(2\sigma_1) - \varphi(\sigma_1)}{\sigma_1} - \frac{\varphi(2\sigma_2) - \varphi(\sigma_2)}{\sigma_2}\right|^2 \leq \frac{2C}{\sigma_1^2} + \frac{2C}{\sigma_2^2}
\]
for every possible $\sigma_1$ and $\sigma_2$, which contradicts the fact that $\varphi \in \Ll_{bid}$.
\end{proof}

We may assume without loss of generality that there is $\delta>0$ such that $\|w_n\|_{\infty} > \delta > 0$ for all $n \in \N$ and so that, for each $n \in \N$, there is $k_n$ such that $\left|w_n(k_n)\right| > \delta$. Now recall the $2$-dimensional UFDD for $Z(\varphi)$ and $Z(\psi)$ defined by $E_n = [(e_n, 0), (0, e_n)]$. Let $P_n$ be the projection onto $E_n$. By considering $P_m \circ S|_{E_n}$, we can see that $S$ is given by a family $(S_{n, m})$ of $2\times 2$-matrices. By the same reasoning of \cite[Proposition 1.c.8]{ClassicalI}, we get that if $S' : Z(\varphi) \rightarrow Z(\psi)$ is defined by
\[
P_m S'|_{E_n} = \begin{cases}
S_{n, k_n} & \mbox{ if } m = k_n, \\
0 & \mbox{ otherwise},
\end{cases}
\]
then $S'$ is bounded. We have
\[
S_{n, k_n} = \begin{pmatrix}
u_n(k_n) & v_n(k_n) \\
0 & w_n(k_n)
\end{pmatrix}
\]
By the last fact of Section \ref{Sec:Background}, by passing to a subsequence, we may suppose that $S_{n, k_n}$ converges to some matrix
\[
S_{\infty} = \begin{pmatrix}
\lambda & \mu \\
0 & \sigma
\end{pmatrix}
\]
with $\left|\sigma\right| \geq \delta$. Finally, by passing to a further subsequence and perturbing the resulting operator, we obtain that $S_{\infty}$ defines a bounded operator from $Z(\varphi)$ into $Z(\psi)$ as we wanted.
\end{proof}
\section{A large family of singular twisted Hilbert spaces}\label{sec:coneable}

The results of the previous sections imply the existence of an infinite-dimensional cone of singular twisted Hilbert spaces by proving that the set of bi-Lipschitz maps from $[0, \infty)$ into $\K$ is coneable. 
Following \cite{Aizpiruetal}, we say that a set $F$ in a vector space is \emph{coneable} (or \emph{convex-coneable}, see, for instance, \cite{BCMRS}) if it contains a convex cone generated by an infinite linearly independent set.

\begin{proposition}\label{prop:g_in_bids}
Let $g : [0, \infty) \rightarrow \K$ be a differentiable function and $\beta > 0$.
If $|g|$ and $|g'|$ are bounded by $M>0$ and $0<\alpha<\frac{1}{M(1+\beta)}$, then $g_{\alpha, \beta} \in \Ll_{bi}$, where
    \[
    g_{\alpha, \beta}(x) := x + \alpha x g(\beta \log x).
    \]
If $g$ is twice differentiable and $\lim_{x \rightarrow \infty} \frac{g''(\beta \log x)}{x} = 0$, then $g_{\alpha, \beta} \in \Ll_{bis}$.
Moreover, if
    \begin{equation}\label{item3}
    \limsup_{n, m \rightarrow \infty} \left|\left(2g\left(\beta \log \log n\right) - g\left(\beta \log \log n^{1/2}\right)\right) - \left(2g\left(\beta \log \log m\right) - g\left(\beta \log \log m^{1/2}\right)\right)\right| > 0,
    \end{equation}
then $g_{\alpha, \beta} \in \Ll_{bids}$.
\end{proposition}
\begin{proof}
It is clear that $g_{\alpha, \beta}\in \mathcal L$ for any $\alpha>0$. We also have that
\[
\abs{x - y} \leq \abs{x + \alpha x g(\beta \log(x)) - y - \alpha y g(\beta \log(y))} + \alpha \abs{x g(\beta \log(x)) - y g(\beta \log(y))}.
\]
As the derivative of $f(x) = x g(\beta \log(x))$ is $f'(x) = g(\beta \log(x)) + \beta g'(\beta \log(x))$, we see that
\[
\abs{x g(\beta \log(x)) - y g(\beta \log(y))} \leq M(1 + \beta) \abs{x - y}.
\]
Therefore, it immediately follows that $g_{\alpha, \beta} \in \Ll_{bi}$ as long as $\alpha<\frac{1}{M(1+\beta)}$.

For the second part, as $g$ is twice differentiable, we have that \[
g''_{\alpha, \beta}(x) = \frac{\alpha \beta}{x}(g'(\beta \log x) + g''(\beta \log x))
\]
so that $g_{\alpha, \beta} \in \mathcal{L}_{bis}$ by the hypothesis $\lim_{x \rightarrow \infty} \frac{g''(\beta \log x)}{x} = 0$. Condition \eqref{item3} implies that $g_{\alpha, \beta} \in \Ll_{bids}$.
\end{proof}

In order to present an example of a function satisfying the conditions of Proposition~\ref{prop:g_in_bids}, we will use Kronecker's Density Theorem (see, for instance, \cite{HW,KH}).
In \cite{FST}, the latter theorem was first used in Lineability (the search for algebraic structures inside non-linear sets \cite{ABPS}). See also \cite{FRST} and \cite{BB} for a better understanding on how this theorem is applied in this context.

\begin{theorem}[Kronecker's Density Theorem] \label{kronecker}
If $\beta_1, \ldots, \beta_n \in \R$ are $\Q$-linearly independent, then 
\begin{equation*}
    \mathcal{D} := \{ (k \beta_1 - [k \beta_1], \ldots, k \beta_n - [k\beta_n]): k \in \N\}
\end{equation*}
is dense in $[0,1]^n$ (here, $[x]$ stands for the integer part of $x$).
\end{theorem}

\begin{example}\label{example}
Take $g(x) = \sin x$ and $\beta' > 0$.
Let $\gamma = \log \frac{1}{2}$ and $\beta=-2\pi \beta' \gamma$.
Of course, $g$ is bounded with bounded derivative and $\lim_{x \rightarrow \infty} \frac{g''(\beta \log x)}{x} = 0$.
Let us check \eqref{item3}. 
By letting $y = \log \log n$ and $z = \log \log m$, we see that we want to estimate
\[
\left|(2g(\beta y) - g(\beta (\gamma + y))) - (2g(\beta z) - g(\beta (\gamma + z)))\right|.
\]
Given $\varepsilon>0$ and going down enough in the real line, we may approximate any real number by one of the form $\log \log n$ up to $\varepsilon$. Thus, since $g$ is Lipschitz, it is enough to show that
\[
\limsup_{y, z \rightarrow \infty} \left|(2g(\beta y) - g(\beta (\gamma + y))) - (2g(\beta z) - g(\beta (\gamma + z)))\right| > 0.
\]
Let $y = -n\gamma$ for some $n$. Then
\begin{eqnarray*}
2g(\beta y) - g(\beta (\gamma + y)) & = & 2 \sin (-\beta n \gamma) - \sin( - \beta n\gamma+\beta\gamma) \\
& = & 2 \sin (- \beta n \gamma) - (\sin(-\beta n \gamma) \cos(\beta \gamma) + \sin(\beta \gamma) \cos(-\beta n \gamma)) \\
& = & \sin(-\beta n \gamma) (2 - \cos(\beta \gamma)) - \sin(\beta \gamma) \cos(-\beta n \gamma) \\
& = & a \sin(-\beta n \gamma) + b \cos(-\beta n \gamma) \\
& = & a \sin \left(2\pi \beta' \gamma^2 n\right) + b \cos\left(2\pi\beta' \gamma^2 n\right)
\end{eqnarray*}
with $a = 2 - \cos(\beta \gamma) \geq 1$ and $b=-\sin(\beta \gamma)$. 
(Note that $a\neq b$.) 
Let us define on $\mathbb R$ the continuous function $L(t)=a \sin(2\pi t) + b \cos(2\pi t)$.
Since the sine and cosine functions are $2\pi$-periodic, Theorem~\ref{kronecker} yields that
    \[
    L(\mathcal D) = \left\{ a \sin \left(2\pi k \beta' \gamma^2\right) + b \cos \left(2\pi k\beta' \gamma^2\right) : k\in \mathbb N \right\}
    \]
is dense in $L([0,1])$.
Thus, we deduce that
\[
\limsup_{n, m \rightarrow \infty} \left|(a \sin(-\beta n \gamma) + b \cos(-\beta n \gamma)) - (a \sin(-\beta m \gamma) + b \cos(-\beta m \gamma))\right|>0,
\]
which is equivalent to
\[
\limsup_{y, z \rightarrow \infty} \left|(2g(\beta y) - g(\beta (\gamma + y))) - (2g(\beta z) - g(\beta (\gamma + z)))\right| > 0.
\]
Therefore, $g = \sin$ satisfies \eqref{item3} with $\beta=-2\pi \beta' \log \frac{1}{2}$ for any $\beta'>0$.
\end{example}
\color{black}

We also need the following technical lemma.

\begin{lemma}\label{technicallemma}
    Let $n\in \mathbb N$, $\lambda_1, \ldots, \lambda_n,\ \alpha_1,\ldots,\alpha_n > 0$, $\alpha=\max\{\alpha_j : 1\leq j \leq n\}$.
    If $\beta_1,\ldots,\beta_n>0$ are $\mathbb Q$-linearly independent and $\alpha \left(1 + 2 \pi \beta \log 2\right)<1$ with $\beta=\max\{ \beta_j : 1\leq j\leq n \}$, then $\sum\limits_{j=1}^n \lambda_j g_{\alpha_j, -2\pi \beta_j\log \frac{1}{2}} \in \mathcal{L}_{bids}$ for $g=\sin$.
\end{lemma}

\begin{proof}
Let $\lambda := \sum\limits_{j=1}^n \lambda_j$, $\gamma=\log \frac{1}{2}$ and $h := \sum\limits_{j=1}^n \lambda_j g_{\alpha_j, -2\pi \beta_j\gamma}$.
It is obvious that $h\in \mathcal L$, so we will first prove that $h\in \mathcal L_{bi}$.
By the reverse triangular inequality, we have that
\begin{equation*}
    |\lambda x - \lambda y| - \left| x \sum_{j=1}^n \lambda_j \alpha_j \sin(-2 \pi \beta_j \gamma \log x) - y \sum_{j=1}^n \lambda_j \alpha_j \sin(-2 \pi \beta_j\gamma \log x) \right|  \leq |h(x) - h(y)|
\end{equation*}
Set $\ell(x) := x f(x)$ with $f(x):=\sum\limits_{j=1}^n \lambda_j \alpha_j \sin(-2 \pi \beta_j \gamma \log x)$. We have that 
\begin{equation*}
    \ell'(x) = \sum_{j=1}^n \lambda_j \alpha_j \sin(-2 \pi \beta_j \gamma \log x) - 2 \pi \gamma \sum_{j=1}^n \lambda_j \alpha_j \beta_j \cos(-2 \pi \beta_j \gamma \log x).
\end{equation*}
Then,
\begin{equation*}
    |\ell'(x)| \leq \alpha \lambda - 2 \pi \alpha \beta \gamma \lambda = \alpha \lambda (1 - 2 \pi \beta \gamma).
\end{equation*}
Therefore, 
\begin{equation*}
\lambda(1 - \alpha(1 - 2 \pi \beta\gamma))|x-y| \leq |h(x) - h(y)|,
\end{equation*}
which implies that $h\in \mathcal L_{bi}$.
Now it is straightforward that $h\in \mathcal L_{bis}$.
It remains to show that $h\in \mathcal L_{bids}$.
Similarly to Example~\ref{example}, by letting $y = \log m$ and $z = \log p$, it is enough to show that
    \begin{equation}\label{conbids}
        \limsup_{y, z \rightarrow \infty} \left|(2h(y) - h(\gamma + y)) - (2h(z) - h(\gamma + z))\right| > 0. 
    \end{equation}
By letting $y=-m\gamma$ for some $m$, we have
\begin{eqnarray*}
2h(y) - h(\gamma + y) & = & \sum_{j=1}^n \left( a_j \sin \left(2\pi \beta_j \gamma^2 m\right) + b_j \cos\left(2\pi\beta_j \gamma^2 m\right) \right),
\end{eqnarray*}
where $a_j=\lambda_j \alpha_j (2-\cos(2\pi \beta_j \gamma))>0$ and $b_j=\lambda_j \alpha_j \sin(2\pi \beta_j \gamma)$ for every $1\leq j \leq n$.
Observe that $a_j\neq b_j$ for any $1\leq j\leq n$.
As above, define the continuous function $L$ on $\mathbb R^n$ as
    \begin{equation*}
    L(t_1, \ldots, t_n) := (a_1\sin(2 \pi t_1)+b_1\cos(2\pi t_1), \ldots, a_n\sin(2 \pi t_n)+b_n\cos(2\pi t_n)).
\end{equation*}
Then, the set 
\begin{equation*}
    L(\mathcal D) = \left\{ \left(a_1 \sin \left(2\pi k \beta_1 \gamma^2\right) + b_1 \cos \left(2\pi k\beta_1 \gamma^2\right),\ldots,a_n \sin \left(2\pi k \beta_n \gamma^2\right) + b_n \cos \left(2\pi k\beta_n \gamma^2\right) \right) : k\in \mathbb N \right\}
\end{equation*}
is dense in $L([0, 1]^n)$ by Theorem~\ref{kronecker}.
As an immediate consequence, we get inequality \eqref{conbids}, which concludes the proof.
\end{proof}

\begin{theorem}\label{bidsconeable} The set $\mathcal{L}_{bids}$ is coneable.
\end{theorem} 

\begin{proof} 


\vspace{0.5cm}

Take $\mathcal{B} \subseteq \left(0,1\right)$ an infinite linearly independent subset of $\R$ over $\Q$.
Observe that, by Lemma~\ref{technicallemma}, it is enough to show that the family $\left\{ g_{10^{-1}, 2 \pi \beta \log 2}: \beta \in \mathcal{B} \right\}$ with $g=\sin$ is linearly independent.


\vspace{0.2cm}
Let $\lambda_1, \ldots, \lambda_n \in \R \setminus \{0\}$ and $\beta_1, \ldots, \beta_n \in \mathcal{B}$ be distinct. By contradiction, assume that 
\begin{equation*}
    \sum_{j=1}^n \lambda_j g_{10^{-1}, 2 \pi \beta_j \log 2} = 0.
\end{equation*}
We have that, for every $x > 0$,
\begin{equation*}
0 = \sum_{j=1}^n \lambda_j g_{10^{-1}, 2 \pi \beta_j \log 2} = \lambda x + \frac{x}{10} \sum_{j=1}^n \lambda_j \sin(2 \pi \beta_j \log 2 \log x).
\end{equation*}
For $x = 1$ we obtain $\lambda = 0$, so that we have, for all $x > 0$,
\begin{equation} \label{eq1}
    \sum_{j=1}^n \lambda_j \sin(2 \pi \beta_j \log 2 \log x) = 0.
\end{equation}
Consider now the function $L: \R^n \rightarrow [-1,1]^n$ given by 
\begin{equation*}
    L(t_1, \ldots, t_n) = (\sin(2 \pi t_1), \ldots, \sin(2 \pi t_n))
\end{equation*}
which is continuous. We have that the set 
\begin{equation*}
    L(\mathcal{D}) = \{ (\sin(2 \pi k \beta_1 \log 2), \ldots, \sin(2 \pi k \beta_n \log 2)): k \in \N \} 
\end{equation*}
is dense in $[-1,1]^n = L([0,1]^n)$ by Theorem \ref{kronecker}.
So, there exists $k \in \N$ such that 
\begin{equation*}
    \sin(2 \pi k \beta_i \log 2 ) > \frac{1}{2} \ \mbox{if} \ \lambda_i > 0 \ \ \ \mbox{and} \ \ \ \sin(2 \pi k \beta_i \log 2 ) < - \frac{1}{2} \ \mbox{if} \ \lambda_i < 0.
\end{equation*}
Take now $x = e^k$. Then, from (\ref{eq1}), we have that 
\begin{equation*}
    \sum_{j=1}^n \lambda_j \sin(2 \pi k \beta_j \log 2 ) = 0.
\end{equation*}
Denote by $\mathcal{J}^+ := \{ j: \lambda_j > 0\}$ and $\mathcal{J}^- := \{ j: \lambda_j < 0 \}$. Thus, we have 
\begin{eqnarray*}
0 &=& \sum_{j=1}^n \lambda_j \sin(2 \pi k \beta_j \log 2 ) \\
&=& \sum_{j \in \mathcal{J}^+} \lambda_j \sin(2 \pi k \beta_j \log 2) + \sum_{j \in \mathcal{J}^-} |\lambda_j| (- \sin(2 \pi k \beta_j \log 2)) \\
&>& \frac{1}{2} \sum_{j \in \mathcal{J}^+} \lambda_j + \frac{1}{2} \sum_{j \in \mathcal{J}^-} |\lambda_j| \\
&=& \frac{1}{2} \sum_{j=1}^n |\lambda_j| > 0
\end{eqnarray*}
which is a contradiction.
\end{proof}
\color{black}

We single out a consequence of the last result.

\begin{corollary}
    The set of bi-Lipschitz maps from $[0, \infty)$ into $\K$ is coneable.
\end{corollary}

Back to twisted Hilbert spaces, we have the following result.

\begin{proposition}\label{notprojcone}
Let $\beta_1 \neq \beta_2$ and assume that $g_{\alpha, \beta_1}, g_{\alpha, \beta_2} \in \Ll_{bids}$ for $g = \sin$. Then $\Omega_{g_{\alpha, \beta_1}}$ is not projectively equivalent to $\Omega_{g_{\alpha, \beta_2}}$.
\end{proposition}
\begin{proof}
We need to show that for any $\beta \neq 0$
\[
\sup_{x > 0} \left|x\right| \left|\sin(\beta_1 \log x) - \beta \sin (\beta_2 \log x)\right| = \infty.
\]
But that is equivalent to
\[
\limsup_{y > 0} \left|\sin(\beta_1 y) - \beta \sin (\beta_2 y)\right| > 0.
\]
Take $y = \frac{2\pi n}{\beta_2}$ so that the second term vanishes. Apply Kronecker's Density Theorem to the first and we are done.
\end{proof}

As a consequence of the previous results, we obtain the following.
\begin{theorem}\label{thm:first_cone}
 There is an infinite-dimensional cone $\mathcal{C}$ of bi-Lipschitz maps such that for every $\varphi \in \mathcal{C}$, $Z(\varphi)$ is a singular twisted Hilbert space. Furthermore, if $\varphi, \psi \in \mathcal{C}$ and $\varphi$ is not a multiple of $\psi$, then $Z(\varphi)$ is incomparable to $Z(\psi)$.
\end{theorem}


\begin{proof}
    Take $\mathcal B\subseteq (0,1)$ an infinite linearly independent subset of $\mathbb R$ over $\mathbb Q$.
    Let $\lambda_1,\ldots,\lambda_n,\mu_1,\ldots,\mu_m>0$ and $\beta_1,\ldots,\beta_n, \gamma_1,\ldots,\gamma_m\in \mathcal B$ distinct, and take $g=\sum_{i=1}^n \lambda_i g_{10^{-1},2\pi\beta_i\log 2}$ and $h=\sum_{j=1}^m \mu_j g_{10^{-1},2\pi \gamma_j\log 2}$.
    We want to show that if $g\neq \alpha h$ for any $\alpha \in \mathbb R$, then $g$ and $h$ are not projectively equivalent.
    By absorbing $\alpha$ into the $\mu_j$'s, it is enough to prove that if $g\neq h$, then $g$ and $h$ are not equivalent.
    
    We have
        \[
        g-h=\left( \sum_{i=1}^n \lambda_i -\sum_{j=1}^m \mu_j \right)x+\frac{x}{10}\left( \sum_{i=1}^n \lambda_i \sin(\beta_i' log x) - \sum_{j=1}^m \mu_j \sin(\gamma_j' \log x) \right)
        \]
    where $\beta_i'=2\pi \beta_i \log 2$ and $\gamma_j'=2\pi \gamma_j \log 2$.
    If $\sum_{i=1}^n \lambda_i -\sum_{j=1}^m \mu_j \neq 0$, then we are done.
    Otherwise, it is sufficient to check that
        \[
        \limsup_{t\to \infty} \left| \sum_{i=1}^n \lambda_i \sin(\beta_i' t)-\sum_{j=1}^m \mu_j \sin(\beta_j't) \right|>0.
        \]
    But the latter follows from the same kind of argument that has been used previously.
\end{proof}

\color{black}

\section{Banach spaces not isomorphic to their conjugate dual and interpolation}\label{sec:interpolation}

Let us recall that, for $\alpha \in \R$, we set $\varphi_{\alpha}(t) = t^{1 + i \alpha}$. As we have mentioned before, Kalton shows in \cite{Kalton_Elementary} that, for each $\alpha \neq 0$, the space $Z(\varphi_{\alpha})$ is not isomorphic to its complex conjugate. It is natural to wonder  what is the connection of the spaces given by Theorem \ref{thm:first_cone} with that kind of behaviour.

We recall that if $X$ is a complex Banach space, its \emph{complex conjugate} is the Banach space $\overline{X}$, which is the same set as $X$ and it has the same sum, same norm with the difference that the multiplication by scalars is replaced by $\lambda \cdot x = \overline{\lambda} x$. The space $\overline{X^*}$ is called the \emph{conjugate dual} of $X$.

The first example of a Banach space that is {\it not} isomorphic to its complex conjugate was given by Bourgain in \cite{Bourgain}. Bourgain used a combination of probability with interpolation to build his example. In \cite{Kalton_Elementary}, Kalton gave an elementary and concrete example through the spaces $Z(\varphi_{\alpha})$. Since $Z(\varphi_{\alpha})^* \simeq Z(-\varphi_{\alpha})$ and $\overline{Z(\varphi_{\alpha})^*} \simeq Z(\varphi_{-\alpha})$, it follows that $Z(\varphi_{\alpha})$ is isomorphic to its dual but not to its conjugate dual.

\vspace{0.2cm}
Suppose that $\varphi, \psi : [0, \infty) \rightarrow \R$ are in the class $\Ll_{bids}$. Then so are $\varphi + i \psi$ and $\varphi - i \psi$. It is easy to see that $\overline{Z(\varphi + i \psi)} \simeq Z(\varphi - i \psi)$. We then obtain the following result.


\begin{theorem}\label{thm:conjugate}
 There is an infinite-dimensional cone $\mathcal{C}$ of bi-Lipschitz maps from $[0, \infty)$ into $\C$ such that for every $\zeta \in \mathcal{C}$, $Z(\varphi)$ is a singular twisted Hilbert space isomorphic to its dual, but not to its conjugate dual. Furthermore, if $\zeta_1, \zeta_2 \in \mathcal{C}$ and $\zeta_1$ is not a multiple of $\zeta_2$ then $Z(\zeta_1)$ is incomparable to $Z(\zeta_2)$. Each space $Z(\zeta)$ is defined through the differential process of complex interpolation of a family of four Banach spaces distributed in arches of the unit circle.
\end{theorem}

To explain the final part of Theorem \ref{thm:conjugate}, let us describe the theory of complex interpolation of Banach spaces. Recall that Bourgain's example mentioned above uses interpolation. Curiously, there is a deep connection between the spaces $Z(\zeta)$ and interpolation.
The reader is referred to  \cite{BerghLofstrom, Calderon1964, COIFMAN1982203, KaltonKothe}.

Let $\zeta = \varphi + i \psi$ with $\varphi, \psi \in \Ll_{bids}$ real-valued. 
There are sequence spaces $X$ and $Y$ such that $\Omega_X$ is projectively equivalent to $\Omega_{\varphi}$, and $\Omega_Y$ is projectively equivalent to $\Omega_{\psi}$ (see \cite[Theorem 6.11]{Singular} for details). 
According to \cite{Butterfly_Lemma}, $\Omega_{\varphi} + i \Omega_{\psi}$ is projectively equivalent to the centralizer on $\ell_2$ induced at $0$ from complex interpolation of the configuration of spaces at the unit circle in Figure~\ref{conf1}.
\begin{figure}[ht]
\begin{center}
\begin{tikzpicture}
   \draw (0,0) circle (2cm);     
   \draw (0,-1.8) -- (0,-2.2);   
   \draw (1.8,0) -- (2.2,0);     
   \draw (0,1.8) -- (0,2.2);     
   \draw (-1.8,0) -- (-2.2,0);   
   \draw (-2,2) node {$X$};
   \draw (2,-2) node {$X^*$};
   \draw (2,2) node {$Y$};
   \draw (-2,-2) node {$Y^*$};
\end{tikzpicture}
\caption{This configuration generates a centralizer projectively equivalent to $\Omega_{\varphi} + i \Omega_{\psi}$.}\label{conf1}
\end{center}
\end{figure}
Let us call that configuration $\mathcal{Z}$. 
The configuration in Figure~\ref{conf2} (which we call $\mathcal{Z}^*$) induces $-\Omega_{\varphi} - i\Omega_{\psi}$.
\begin{figure}[ht]
\begin{center}
\begin{tikzpicture}
   \draw (0,0) circle (2cm);     
   \draw (0,-1.8) -- (0,-2.2);   
   \draw (1.8,0) -- (2.2,0);     
   \draw (0,1.8) -- (0,2.2);     
   \draw (-1.8,0) -- (-2.2,0);   
   \draw (-2,2) node {$X^*$};
   \draw (2,-2) node {$X$};
   \draw (2,2) node {$Y^*$};
   \draw (-2,-2) node {$Y$};
\end{tikzpicture}
\caption{This configuration generates a centralizer projectively equivalent to $-\Omega_{\varphi} - i\Omega_{\psi}$.}\label{conf2}
\end{center}
\end{figure}
Notice that the configuration $\mathcal{Z}^*$ is obtained from the configuration $\mathcal{Z}$ by the map $z \mapsto -z$. 
Since $f$ is analytic if and only if $z \mapsto f(-z)$ is analytic, we obtain an isometry between $\mathscr{F}\left(\mathcal Z\right)$ and $\mathscr{F}\left(\mathcal{Z}^*\right)$, which in turn induces an isometry between $Z(\varphi + i \psi)$ and $Z(-\varphi - i \psi)$.
Lastly, $\Omega_{\varphi} - i\Omega_{\psi}$ is obtained through the configuration of spaces in Figure~\ref{conf3} (which we call $\overline{\mathcal{Z}}$)
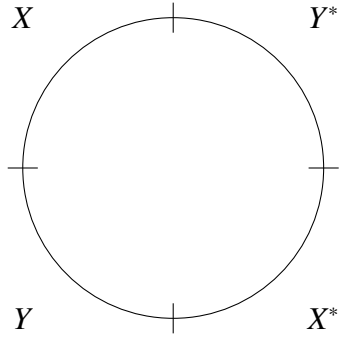
\begin{figure}[ht]
\begin{center}
\begin{tikzpicture}
   \draw (0,0) circle (2cm);     
   \draw (0,-1.8) -- (0,-2.2);   
   \draw (1.8,0) -- (2.2,0);     
   \draw (0,1.8) -- (0,2.2);     
   \draw (-1.8,0) -- (-2.2,0);   
   \draw (-2,2) node {$X$};
   \draw (2,-2) node {$X^*$};
   \draw (2,2) node {$Y^*$};
   \draw (-2,-2) node {$Y$};
\end{tikzpicture}
\caption{This configuration generates a centralizer projectively equivalent to $\Omega_{\varphi} - i \Omega_{\psi}$.}\label{conf3}
\end{center}
\end{figure}
and we see that to obtain a natural isomorphism between $\mathscr{F}(\mathcal{Z})$ and $\mathscr{F}(\overline{\mathcal{Z}})$ we should have that $f$ is analytic if and only if $z \mapsto f(\overline{z})$ is analytic. Therefore, the fact that $Z(\zeta)$ is isomorphic to its dual but not to its conjugate dual is naturally related to the fact that $z \mapsto -z$ is analytic, while $z \mapsto \overline{z}$ is not.

\section{A characterization of the Kalton-Peck space}\label{sec:characterization}

As we have seen throughout the paper, we have built a large family of spaces $Z(\varphi)$ sharing the properties of $Z_2$, but not isomorphic to $Z_2$. In this section, we present a characterization of the Kalton-Peck space among all twisted Hilbert spaces of the form $Z(\varphi)$. This provides a partial answer to a conjecture by Cabello and Castillo (see \cite[page 517]{Homological}).

Recall that if $X$ is a Banach space with normalized basis $(x_n)$ and $(x_n)$ is equivalent to all of its normalized block bases, then, $(x_n)$ is equivalent to the canonical basis of either $c_0$ or of some $\ell_p$ space ($1 \leq p < \infty)$ (see \cite{Zippin}).

The Kalton-Peck space satisfies the following similarity property (see \cite[page 517]{Homological}). Let $X$ be a twisted Hilbert space. Given a normalized block basis $(u_n)$ of the canonical basis of $\ell_2$, let $U = [u_n]_n = \overline{\mbox{span}} \{u_n : n \in \N\}$ and let $T_U : \ell_2 \rightarrow U$ be given by $T_U(x) = \sum x_n u_n$. Let
\[
U(X) = \overline{\mbox{span}}\{(y, x) \in X : y, x \in \mbox{span}\{u_n : n \in \N\}\}
\]
We say that $X$ is \emph{self-similar with respect to $(u_n)$} if there is an isomorphism $T : X \rightarrow U(X)$ making the following diagram commutative
\[
\xymatrix{0 \ar[r] & \ell_2 \ar[r]\ar[d]^{T_U} & X \ar[r]\ar[d]^{T} & \ell_2 \ar[r]\ar[d]^{T_U} & 0 \\
0 \ar[r] & U \ar[r] & U(X) \ar[r] & U \ar[r] & 0}
\]
We say that $X$ is \emph{self-similar} if it is self-similar with respect to every normalized block basis $(u_n)$ of the canonical basis of $\ell_2$.
 
\vspace{0.2cm}
It is clear that $X = \ell_2 \oplus \ell_2$ is self-similar. As mentioned, $Z_2$ is self-similar as well. In the language of \cite{Symplectic, Structure_Rochberg}, the operator $T$ appearing above for $X = Z_2$ is a \emph{block operator} (see also \cite{Group_Actions}). Our goal here is to prove that the Kalton-Peck space $Z_2$ is the unique nontrivial self-similar twisted Hilbert space of the form $Z(\varphi)$. The proof starts by essentially showing that if $X = Z(\varphi)$ is self-similar, then we can take the operator $T$ to be a block operator.

\begin{lemma}\label{lem:L_exists}
The twisted Hilbert space $Z(\varphi)$ is self-similar with respect to $(u_n)$ if and only if there is a linear map $L : \ell_2 \rightarrow U$ and $C > 0$ such that
\[
\|T_U(\Omega_{\varphi}x) - \Omega_{\varphi}(T_U x) + Lx\|_2 \leq C \|x\|_2
\]
for every $x \in c_{00}$. Furthermore, $\|T\| \leq 1 + C$.
\end{lemma}
\begin{proof}
Suppose that $Z(\varphi)$ is self-similar with respect to $(u_n)$ and let $T$ be a witness for $(u_n)$. It follows that $T$ necessarily has the form $T(y, x) = (T_U y + Lx, T_U x)$. Let $x \in c_{00}$. We have that
\begin{eqnarray*}
\|T_U(\Omega_{\varphi}x) - \Omega_{\varphi}(T_U x) + Lx\|_2 \leq \|T(\Omega_{\varphi} x, x)\|_{Z(\varphi)} \leq \|T\| \|x\|_2.
\end{eqnarray*}
On the other hand, if
\[
\|T_U(\Omega_{\varphi}x) - \Omega_{\varphi}(T_U x) + Lx\|_2 \leq C \|x\|_2
\]
for every $x \in c_{00}$, define $T(y, x) = (T_U y + Lx, T_U x)$ for $x, y \in c_{00}$. We have
\begin{eqnarray*}
\|T(y, x)\|_{Z(\varphi)} & = & \|T_U(y) + Lx - \Omega_{\varphi}(T_Ux)\|_2 + \|T_Ux\|_2 \\
& \leq & \|(y, x)\|_{Z(\varphi)} + \|T_U(\Omega_{\varphi}x) - \Omega_{\varphi}(T_U x) + Lx\|_2 \\
& \leq & (1 + C) \|(y, x)\|_{Z(\varphi)}
\end{eqnarray*}
as we wanted to prove.
\end{proof}

\begin{lemma}
Let $u_n = \sum\limits_{k=p_n + 1}^{p_{n+1}} u_n(k) e_k$. If $Z(\varphi)$ is self-similar with respect to $(u_n)$ we may replace the map $L$ from Lemma \ref{lem:L_exists} by $L'$ given by
\[
L'(e_n) = \sum\limits_{k=p_n + 1}^{p_{n+1}} \varphi\left(\log \frac{1}{\left|u_n(k)\right|}\right) u_n(k) e_k
\]
replacing $C$ by $2C$.
\end{lemma}
\begin{proof}
Let $L$ be given by Lemma \ref{lem:L_exists} and let $\mathcal{U} = \{+1,-1\}^{\N}$ be the group of units of real valued $\ell_{\infty}$. Following \cite[Lemma 3.15]{Singular}, define
\[
\Lambda(x) = \int_{\mathcal{U}} S_U(v) L(vx) dm
\]
for $x \in \ell_2$, where $S_U(e_n) = \sum\limits_{j = p_n + 1}^{p_{n+1}} e_j$ and $m$ is a left invariant finitely additive mean on $\mathcal{U}$. Notice that $\Lambda$ is linear and for $x \in c_{00}$
\begin{eqnarray*}
\|T_U(\Omega_{\varphi}x) - \Omega_{\varphi}(T_U x) + \Lambda(x)\|_2 & = & \left\| \int_{\mathcal{U}} S_U(v) [T_U(\Omega_{\varphi}(vx)) - \Omega_{\varphi}(T_U (vx)) + L(vx)] dm \right\|_2 \\
& \leq & \|T_U(\Omega_{\varphi}x) - \Omega_{\varphi}(T_U x) + L(x)\|_2 \\
& \leq & C \|x\|_2.
\end{eqnarray*}
Also, if $w \in \mathcal{U}$, then
\[
\Lambda(w x) = \int_{\mathcal{U}} S_U(v) L(v wx) dm(v) = \int_{\mathcal{U}} S_U(w) S_U(vw) L(vwx) dm(v) = S_U(w) \Lambda(x). 
\]
Thus $\Lambda(e_n) = \Lambda(e_n e_n) = S_U(e_n) \Lambda(e_n)$ for every $n$. It follows that the support of $\Lambda(e_n)$ is contained in the support of $u_n$, so that we may write
\[
\Lambda(e_n) = \sum\limits_{k= p_n + 1}^{p_{n+1}} \lambda_n(k) u_n(k) e_k.
\]
Since
\[
T_U(\Omega_{\varphi}e_n) - \Omega_{\varphi}(T_U e_n) = - \sum\limits_{k= p_n + 1}^{p_{n+1}} u_n(k) \varphi\left(\log \frac{1}{\left|u_n(k)\right|}\right) e_k
\]
we obtain
\[
\left\|\sum\limits_{k= p_n + 1}^{p_{n+1}} u_n(k) \left(\lambda_n(k) -  \varphi\left(\log \frac{1}{\left|u_n(k)\right|}\right)\right) e_k\right\|_2 \leq C
\]
for every $n$. Finally, let $x \in c_{00}$ and define
\[
L'(x) = \sum\limits_n x_n \sum\limits_{k= p_n + 1}^{p_{n+1}} u_n(k) \varphi\left(\log \frac{1}{\left|u_n(k)\right|}\right) e_k.
\]
Let us see that we may replace $L$ by $L'$:
\begin{eqnarray*}
\|T_U(\Omega_{\varphi}x) - \Omega_{\varphi}(T_U x) + L'(x) \|_2 & \leq & C \|x\|_2 + \|\Lambda(x) - L'(x)\|_2 \\
& = & C \|x\|_2 + \left\|\sum\limits_n x_n \sum\limits_{k= p_n + 1}^{p_{n+1}} u_n(k) \left(\lambda_n(k) -  \varphi\left(\log \frac{1}{\left|u_n(k)\right|}\right)\right) e_k\right\|_2 \\
& \leq & 2 C \|x\|_2
\end{eqnarray*}
since the $u_n$'s are disjointly supported.
\end{proof}

According to Lemma \ref{lem:L_exists}, associated to $L'$ there is a witness that $Z(\varphi)$ is self-similar with respect to $(u_n)$. Let us denote that witness by $T_{U, U}$ and call it \emph{the canonical witness for $(u_n)$}.

\begin{lemma}
If $Z(\varphi)$ is self-similar, then there is a constant $C$ independent of the normalized block basis $(u_n)$ such that for every $x \in c_{00}$ we have
\begin{equation*} 
\left\| \sum \limits_n x_n \sum\limits_{k= p_n + 1}^{p_{n+1}} u_n(k) \left[\varphi\left(\log \frac{\|x\|_2}{\left|x_n\right|}\right) + \varphi\left(\log \frac{1}{\left|u_n(k)\right|}\right) - \varphi\left(\log \frac{\|x\|_2}{\left|u_n(k) x_n\right|}\right)\right] e_k \right\|_2 \leq C \|x\|_2.
\end{equation*} 
\end{lemma}
\begin{proof}
Notice that the expression inside the norm is simply $T_U(\Omega_{\varphi}x) - \Omega_{\varphi}(T_U x) + L'(x)$, where $L'$ is given by the previous lemma. Suppose the conclusion is false. That implies that for each $N$ we are able to find a normalized block basis $(u_n^N)_n$ of $\ell_2$ for which $\|T_{U^N, U^N}\| > N$, where $U^N = [u_n^N]_n$.

We may see $(u_n^N)_{n, N}$ as a normalized block basis for $\ell_2(\N \times \N)$. Use a bijection between $\N$ and $\N \times \N$ to transport this block basis to $\ell_2$, call it $(v_n)$ and let $V = [v_n]$. Notice that $T_{V, V}$ is an amalgamation of the maps $T_{U^N, U^N}$, and therefore should be unbounded, contradicting the fact that $Z(\varphi)$ is self-similar.
\end{proof}

\begin{theorem}
If $Z(\varphi)$ is nontrivial and self-similar, then it is projectively equivalent (and therefore isomorphic) to the Kalton-Peck space $Z_2$.
\end{theorem}
\begin{proof}
Given $N \geq 1$, let $
u_n^N = \frac{1}{\sqrt{N}} \sum\limits_{j=(n-1)N + 1}^{nN} e_j$
and notice that $(u_n^N)_n$ is a normalized block basis for $(\ell_2)$. Take $M \geq 1$ and $x = \sum\limits_{j=1}^M e_j$ in the previous lemma to get that
\[
\left|\varphi\left(\log \sqrt{M}\right) + \varphi\left(\log \sqrt{N}\right) - \varphi\left(\log \sqrt{NM}\right)\right| \leq C
\]
for every $N, M \geq 1$. Since $\varphi$ is Lipschitz, we obtain that
\[
\sup_{a, b > 0} \left|\varphi(a + b) - \varphi(a) - \varphi(b)\right| < \infty.
\]
Following \cite{Hyers}, there exists $c = \lim_{t \rightarrow \infty} \frac{\varphi(t)}{t}$ and
\[
\sup_{t > 0} \left|\varphi(t) - ct\right| < \infty
\]
and therefore $Z(\varphi)$ is projectively equivalent to $Z_2$.
\end{proof}

As two-thirds of the authors of \cite{Homological}, we conjecture that the Kalton-Peck space $Z_2$ is, up to isomorphism, the only nontrivial self-similar twisted Hilbert space.



\vspace{0.2cm} 
\noindent 
{\bf Acknowledgments}: The authors express their gratitude to Jesús Castillo for his valuable feedback on an earlier draft of this manuscript. His diverse comments and suggestions have significantly improved the final version of this paper.

\vspace{0.2cm} 
\noindent 
{\bf Funding}: Willian Corrêa was supported by São Paulo Research Foundation (FAPESP), grants 2016/25574-8, 2021/13401-0, 2023/06973-2 and 2023/12916-1, and National Council for Scientific and Technological Development - CNPq - Brasil, grant 304990/2023-0. Sheldon Dantas was supported by the Spanish AEI Project PID2019 - 106529GB - I00 / AEI / 10.13039/501100011033, by
 Generalitat Valenciana project CIGE/2022/97 and 
by the grant PID2021-122126NB-C33 funded by 
MICIU/AEI/10.13039/501100011033 and by ERDF/EU.

\bibliographystyle{amsplain}
\bibliography{refs}

\providecommand{\bysame}{\leavevmode\hbox to3em{\hrulefill}\thinspace}
\providecommand{\MR}{\relax\ifhmode\unskip\space\fi MR }
\providecommand{\MRhref}[2]{%
  \href{http://www.ams.org/mathscinet-getitem?mr=#1}{#2}
}
\providecommand{\href}[2]{#2}
\begin{thebibliography}{10}

\bibitem{Aizpiruetal}
A.~Aizpuru, C.~Pérez-Eslava, F.~Garcia-Pacheco, and J.~Seoane-Sepúlveda,
  \emph{Lineability and coneability of discontinuous functions on
  $\mathbb{R}$}, Publ. Math. Debrecen \textbf{72} (2008), 129--139.

\bibitem{ABPS}
R.~M. Aron, L.~Bernal~Gonz\'alez, D.~M. Pellegrino, and J.~B.
  Seoane~Sep\'ulveda, \emph{Lineability: the search for linearity in
  mathematics}, Monographs and Research Notes in Mathematics, CRC Press, Boca
  Raton, FL, 2016.

\bibitem{Benyamini1998geometric}
Y.~Benyamini and J.~Lindenstrauss, \emph{Geometric nonlinear functional
  analysis. {V}ol. 1}, American Mathematical Society Colloquium Publications,
  vol.~48, American Mathematical Society, Providence, RI, 2000.

\bibitem{BerghLofstrom}
J.~Bergh and J.~L\"ofstr\"om, \emph{Interpolation spaces. {A}n introduction},
  Grundlehren der Mathematischen Wissenschaften, vol. 223, Springer-Verlag,
  Berlin-New York, 1976.

\bibitem{BB}
L.~Bernal-Gonz\'alez and A.~Bonilla, \emph{Prescribed sets of continuity:
  {B}aire-{K}uratowski theorem, derivatives, and lineability}, J. Math. Anal.
  Appl. \textbf{515} (2022), no.~1.

\bibitem{BCMRS}
L.~Bernal-Gonz\'alez, M.~C. Calder\'on-Moreno, G.~A. Mu\~noz Fern\'andez, D.~L.
  Rodr\'iguez-Vidanes, and J.~B. Seoane-Sep\'ulveda, \emph{Coneability of
  anti-{F}ubini functions and other lineability properties}, Results Math.
  \textbf{79} (2024), no.~2.

\bibitem{Bourgain}
J.~Bourgain, \emph{Real isomorphic complex {B}anach spaces need not be complex
  isomorphic}, Proc. Amer. Math. Soc. \textbf{96} (1986), no.~2, 221--226.

\bibitem{HigherOrderUs}
F.~Cabello~S\'anchez, J.~M.~F. Castillo, and W.~H.~G. Corr\^ea, \emph{Higher
  order derivatives of analytic families of {B}anach spaces}, Studia Math.
  \textbf{272} (2023), no.~3, 245--297.

\bibitem{Cabello_no_singular}
F.~Cabello~Sánchez, \emph{There is no strictly singular centralizer on
  {$L_p$}}, Proc. Amer. Math. Soc. \textbf{142} (2014), no.~3, 949--955.

\bibitem{Cabello_Factorization}
\bysame, \emph{Factorization in {L}orentz spaces, with an application to
  centralizers}, J. Math. Anal. Appl. \textbf{446} (2017), no.~2, 1372--1392.

\bibitem{Homological}
F.~Cabello~Sánchez and J.~M.~F. Castillo, \emph{Homological methods in
  {B}anach space theory}, Cambridge Studies in Advanced Mathematics, vol. 203,
  Cambridge University Press, Cambridge, 2023.

\bibitem{Cabello_sso}
F.~Cabello~Sánchez, J.~M.~F. Castillo, and J.~Suárez, \emph{On strictly
  singular nonlinear centralizers}, Nonlinear Anal. \textbf{75} (2012), no.~7,
  3313--3321.

\bibitem{Calderon1964}
A.~Calderón, \emph{Intermediate spaces and interpolation, the complex method},
  Studia Math. \textbf{24} (1964), 113--190.

\bibitem{Castillo}
J.~M.~F. Castillo, \emph{The freewheeling twisting of {H}ilbert spaces}, The
  mathematical legacy of Victor Lomonosov---operator theory \textbf{2} (2020),
  43--66.

\bibitem{Symmetries}
J.~M.~F. Castillo, W.~H.~G. Corr\^ea, V.~Ferenczi, and M.~Gonz\'alez,
  \emph{Interpolator symmetries and new {K}alton-{P}eck spaces}, Results Math.
  \textbf{79} (2024), no.~3.

\bibitem{CCFG}
J.~M.~F. Castillo, W.~G. Corrêa, V.~Ferenczi, and M.~González,
  \emph{Differential processes generated by two interpolators}, RACSAM
  \textbf{114} (2020), paper 183.

\bibitem{Symplectic}
J.~M.~F. Castillo, W.~Cuellar, M.~Gonz\'alez, and R.~Pino, \emph{On symplectic
  {B}anach spaces}, Rev. Real Acad. Cienc. Exactas Fis. Nat. Ser. A-Mat.
  \textbf{117} (2023), no.~56.

\bibitem{Group_Actions}
J.~M.~F. Castillo and V.~Ferenczi, \emph{Group actions on twisted sums of
  {B}anach spaces}, Bull. Malays. Math. Sci. Soc. \textbf{46} (2023), no.~135.

\bibitem{Singular}
J.~M.~F. Castillo, V.~Ferenczi, and M.~González, \emph{Singular twisted sums
  generated by complex interpolation}, Trans. Amer. Math. Soc. \textbf{369}
  (2017), no.~7, 4671--4708.

\bibitem{Structure_Rochberg}
J.~M.~F. Castillo, M.~Gonz\'alez, and Pino R., \emph{The structure of rochberg
  spaces}, J. Funct. Anal. \textbf{287} (2024), no.~4, 110489.

\bibitem{3SP}
J.~M.~F. Castillo and M.~González, \emph{Three-space problems in {B}anach
  space theory}, Lecture Notes in Mathematics, vol. 1667, Springer-Verlag,
  Berlin, 1997.

\bibitem{Quasilinear_Inversion}
\bysame, \emph{Quasilinear duality and inversion in {B}anach spaces}, Proc. R.
  Soc. Edinb. A: Math. (2024), 1–17.

\bibitem{OperatorsZ2}
J.~M.~F. Castillo, M.~González, and R.~Pino, \emph{Operators on the
  {K}alton-{P}eck space $\text{Z}_2$}, Available on arXiv (2022).

\bibitem{Butterfly_Lemma}
J.~M.~F. Castillo and D.~Morales, \emph{The butterfly lemma}, Nonlinear Anal.
  \textbf{215} (2022).

\bibitem{CMS}
J.~M.~F. Castillo, D.~Morales, and J.~{Suárez de la Fuente}, \emph{Derivation
  of vector-valued complex interpolation scales}, J. Math. Anal. Appl.
  \textbf{468} (2018), no.~1, 461--472.

\bibitem{COIFMAN1982203}
R.~R. Coifman, M.~Cwikel, R.~Rochberg, Y.~Sagher, and G.~Weiss, \emph{A theory
  of complex interpolation for families of banach spaces}, Advances in
  Mathematics \textbf{43} (1982), no.~3, 203--229.

\bibitem{TypeCotype}
W.~H.~G. Corrêa, \emph{Type, cotype and twisted sums induced by complex
  interpolation}, J. Funct. Anal. \textbf{274} (2018), no.~3, 797--825.

\bibitem{FRST}
J.~Fern\'andez-S\'anchez, D.~L. Rodr\'iguez-Vidanes, J.~B. Seoane-Sep\'ulveda,
  and W.~Trutschnig, \emph{Lineability, differentiable functions and special
  derivatives}, Banach J. Math. Anal. \textbf{15} (2021), no.~1.

\bibitem{FST}
J.~Fern\'andez-S\'anchez, J.~B. Seoane-Sep\'ulveda, and W.~Trutschnig,
  \emph{Lineability, algebrability, and sequences of random variables}, Math.
  Nachr. \textbf{295} (2022), no.~5, 861--875.

\bibitem{HW}
G.~H. Hardy and E.~M. Wright, \emph{An introduction to the theory of numbers},
  sixth ed., Oxford University Press, Oxford, 2008, Revised by D. R.
  Heath-Brown and J. H. Silverman, With a foreword by Andrew Wiles.

\bibitem{Hyers}
D.~H. Hyers, \emph{On the stability of the linear functional equation}, Proc.
  Nat. Acad. Sci. U.S.A. \textbf{27} (1941), 222--224.

\bibitem{KaltonKothe}
N.~J. Kalton, \emph{Differentials of complex interpolation processes for
  {K}\"othe function spaces}, Trans. Amer. Math. Soc. \textbf{333} (1992),
  no.~2, 479--529.

\bibitem{Kalton_Elementary}
\bysame, \emph{An elementary example of a {B}anach space not isomorphic to its
  complex conjugate}, Canad. Math. Bull. \textbf{38} (1995), no.~2, 218--222.

\bibitem{Kalton_Unconditional}
\bysame, \emph{Twisted {H}ilbert spaces and unconditional structure}, J. Inst.
  Math. Jussieu \textbf{2} (2003), no.~3, 401--408.

\bibitem{KP}
N.~J. Kalton and N.~T. Peck, \emph{Twisted sums of sequence spaces and the
  three space problem}, Trans. Amer. Math. Soc. \textbf{255} (1979), 1--30.

\bibitem{KH}
A.~Katok and B.~Hasselblatt, \emph{Introduction to the modern theory of
  dynamical systems}, Encyclopedia of Mathematics and its Applications,
  vol.~54, Cambridge University Press, Cambridge, 1995, With a supplementary
  chapter by Katok and Leonardo Mendoza.

\bibitem{ClassicalI}
J.~Lindenstrauss and L.~Tzafriri, \emph{Classical {B}anach spaces. {I}},
  Ergebnisse der Mathematik und ihrer Grenzgebiete [Results in Mathematics and
  Related Areas], vol.~92, Springer-Verlag, Berlin-New York, 1977, Sequence
  spaces.

\bibitem{ELP}
G.~Pisier, J.~Lindenstrauss, and P.~Enflo, \emph{On the ``three space
  problem''}, Math. Scand. \textbf{36} (1975), no.~2, 199--210.

\bibitem{JesusKalton-Peck}
J.~Suárez de~la Fuente, \emph{The {K}alton-{P}eck space as a spreading model},
  Math. Scand. \textbf{130} (2024), no.~2, 161--197.

\bibitem{Zippin}
M.~Zippin, \emph{On perfectly homogeneous bases in {B}anach spaces}, Israel J.
  Math. \textbf{4} (1966), 265--272.

\end{thebibliography}

\end{document}